\documentclass[ctagsplt,12pt]{amsart}
\RequirePackage{newlfont}
\usepackage{amscd}
\usepackage{amsmath}
\usepackage{amsfonts}
\usepackage{amssymb}
\usepackage{enumerate}
\usepackage{graphicx}
\usepackage{latexsym}
\usepackage{color}

\numberwithin{equation}{section}

\begin{document}
	
	\title[Modular Uniform Convexity]{Modular uniform convexity structures and applications to boundary value problems with non-standard growth}
	\author[ M. A. Khamsi \& O. M\'{e}ndez]{Mohamed  A. Khamsi \&  Osvaldo M\'{e}ndez}

	\address{Mohamed A. Khamsi\\ Department of Applied Mathematics and Sciences, Khalifa University, Abu Dhabi, UAE.}
	\email{mohamed.khamsi@ku.ac.ae}
	\address{Osvaldo M\'{e}ndez\\Department of Mathematical Sciences, The University of Texas at El Paso, El Paso, TX 79968, U.S.A.}
	\email{osmendez@utep.edu}
	
	\begin{abstract}
		We establish the existence and uniqueness of the solution to the Dirichlet problem for the variable exponent $p$-Laplacian on a bounded, smooth domain $\Omega \subset {\mathbb R}^n$, where the boundary datum belongs to $W^{1,p}(\Omega)$. Our analysis considers a continuous and bounded exponent $p$ satisfying $1<\inf\limits_{x\in \Omega}p(x)$ and $\sup\limits_{x\in \Omega}p(x)<\infty $, and is based on the uniform convexity of the Dirichlet integral, which is highly non trivial and in the variable exponent case is not related to the uniform convexity of the Sobolev norm.

	\end{abstract}
	\maketitle
	
	\newtheorem{theorem}{Theorem}[section]
	\newtheorem{acknowledgement}{Acknowledgement}
	\newtheorem{algorithm}{Algorithm}
	\newtheorem{axiom}{Axiom}[section]
	\newtheorem{case}{Case}
	\newtheorem{claim}{Claim}
	\newtheorem{conclusion}{Conclusion}
	\newtheorem{condition}{Condition}
	\newtheorem{conjecture}{Conjecture}
	\newtheorem{corollary}{Corollary}[section]
	\newtheorem{criterion}{Criterion}
	\newtheorem{definition}{Definition}[section]
	\newtheorem{example}{Example}[section]
	\newtheorem{exercise}{Exercise}
	\newtheorem{lemma}{Lemma}[section]
	\newtheorem{notation}{Notation}
	\newtheorem{problem}{Problem}
	\newtheorem{proposition}{Proposition}[section]
	\newtheorem{remark}{Remark}[section]
	\newtheorem{solution}{Solution}
	\newtheorem{summary}{Summary}

	\thispagestyle{empty}
	\section{Introduction}
	We prove the solvability of the non-homogeneous Dirichlet problem for the variable exponent $p(x)$-Laplacian, with boundary datum in the Sobolev space $W^{1,p(\cdot)}(\Omega)$. In the sequel, for notational simplicity and without further notice, a variable exponent $p(x)$ will simply be denoted by $p$; in particular in the notation $W^{1,p}(\Omega)$ it will be understood that the exponent $p$ is variable, unless specifically indicated otherwise. We refer the reader to Definition \ref{ves-continuous} for the specifics. The domain $\Omega \subset \mathbb{R}^n$ is assumed to be bounded and regular.
	Though variable exponent spaces were first introduced in the early 1930's, the theory received fresh impetus when it was discovered that such spaces are the natural habitat of the solutions of differential equations with non-standard growth. As a reference starting point (notwithstanding the existence of prior works) we mention the boundary value problems introduced in \cite{KR}.\\
	A significant advance in this direction was obtained in \cite{FZ} (see also \cite{CF}): the authors succeeded in proving the existence and uniqueness of the solution of the homogeneous Dirichlet problem for the variable exponent $p(x)$-Laplacian, namely
	\begin{equation}\label{DP}
		\begin{cases}
			\Delta_p(u)=\text{div}\left(|\nabla u|^{p-2}\nabla u\right)=f\,\,\,\text{in}\,\,\, \Omega\\
			u|_{\partial \Omega}= 0,
		\end{cases}
	\end{equation}
	where $\Omega$ is a bounded, regular domain, the exponent $p=p(x)$ is bounded away from $1$ and $\infty$ and $f$ is a suitable Carath\'{e}odory function.
	Since then, a diverse variety of $p(x)$-Laplacian type boundary value problems with zero boundary data were considered and solved using powerful nonlinear techniques.Though it is impossible to provide a complete list of the literature, we wish to highlight a few works in this direction\\
	The problem
	\begin{equation*}
		\begin{cases}
			-\Delta_p(u)+a(x)|u|^{p-2}u=f(x,u)\,\,\,\text{in}\,\,\, \Omega\\
			u|_{\partial \Omega}= 0
		\end{cases}
	\end{equation*}
	was studied in \cite{CF}.
	The general homogeneous boundary value problem
	\begin{equation*}
		\begin{cases}
		\sum\limits_iD_i\left(a_i(x,u)|D_iu|^{p-2}D_iu\right)=f(x,u)\,\,\,\text{in}\,\,\, \Omega\\
			u|_{\partial \Omega}= 0
		\end{cases}
	\end{equation*}
	was solved in \cite{AS}.
	In \cite{AC} the mixed problem
	\begin{equation*}
	\begin{cases}
			-\Delta_p(u)+c(x,u)|u|^{\sigma -2}u=f\,\,\,\text{in}\,\,\, \Omega\\
		u|_{\Gamma_0}= 0,\\
		\left(\partial u + b(s,u)|u|^{\gamma-2}u\right)|_{\Gamma 1}=g
	\end{cases}
	\end{equation*}
is considered, where the boundary of $\Omega$ is the disjoint union of $\Gamma_0$ and $\Gamma_1$, $\partial$ stands for an oblique derivative, $f$ is a Carath\'{e}odory function and $c,\sigma, \gamma$ are subject to additional conditions (we refer the reader to the article for the specifics). It is worth underlying the fact that some control is required here on the modulus of continuity of the variable exponent. Notice that the Dirichlet condition in this problem is homogeneous. \\
In a similar spirit, the boundary value problem with homogeneous boundary data 
	\begin{equation*}
	\begin{cases}
		-\Delta_p(u)=a(x)|u|^{r-2}u+f(x,u)\,\,\,\text{in}\,\,\, \Omega\\
		u|_{\partial \Omega}= 0,
	\end{cases}
\end{equation*}
is studied in \cite{YW}; here $p$ and $r$ are variable exponents. Under suitable conditions, a sequence of nontrivial weak solutions exists for this problem.\\
Later, in \cite{SU} entropy solutions are shown to exist for the homogeneous boundary value problem
	\begin{equation*}
	\begin{cases}
		\text{div}\left(a(x,\nabla u)\right)=f\,\,\,\text{in}\,\,\, \Omega\\
		u|_{\partial \Omega}= 0,
	\end{cases}
\end{equation*}
under suitable nonstandard conditions on $a(x,z)$.\\Related homogeneous Dirichlet problems in variable exponent Sobolev spaces were studied in \cite{MR1,MR2}.
Systems of differential equations with nonstandard growth have also been considered, for example in  \cite{Z1} (see also \cite{H}), where the homogeneous Dirichlet problem (here $p=p(x)$)
	\begin{equation*}
	\begin{cases}
		-\Delta_p(u)=\lambda f(x,v)\,\,\,\text{in}\,\,\, \Omega\\
		-\Delta_p(v)=\lambda g(x,u)\,\,\,\text{in}\,\,\, \Omega\\
		u|_{\partial \Omega}=v|_{\partial \Omega}= 0,
	\end{cases}
\end{equation*}
is shown to have positive solutions, obviously with some assumptions on the parameters. Parabolic versions have also been targeted, see for example \cite{ZZ}. More general systems with homogeneous Dirichlet boundary conditions were considered in \cite{AS}. In all cases, only zero boundary conditions are considered.\\
The statement of a Dirichlet problem with non-homogeneous boundary condition
\begin{equation}\label{DPNH}
	\begin{cases}
		\Delta_p(u)=\text{div}\left(|\nabla u|^{p-2}\nabla u\right)=0\,\,\,\text{in}\,\,\, \Omega\\
		u|_{\partial \Omega}= g,
	\end{cases}
\end{equation}
calls for the clear definition of the sense in which the boundary condition is to be interpreted. If the boundary data $g$ is continuous, then the equality can be thought of pointwise. In applications, however, continuity is rarely found and more general boundary data are desirable. For example, if $g$ lies in some Lebesgue space the equality is to be understood in the sense of non-tangential convergence \cite{DaKe}. If $g\in W^{1,p}(\Omega)$, then a weak solution $w$ to the differential equation in problem (\ref{DPNH}) is said satisfy the boundary condition if $w-g\in W^{1,p}_0(\Omega)$.\\
If $p$ is constant in $\Omega$ and $g\in W^{1,p}(\Omega)$, a possible approach to make sense of the boundary condition \cite{BR} is to exploit the uniform convexity of the Sobolev norm $u\rightarrow \||\nabla u|\|_{L^p}$ to find a minimizer $u_0\in W_0^{1,p}(\Omega)$ of the Dirichlet energy functional
\begin{equation}\label{dii}
W^{1,p}_0(\Omega)\ni u\rightarrow \int\limits_{\Omega}|\nabla (u-g)|^pdx
\end{equation}
and to observe that the Fr\'{e}chet derivative of the functional (\ref{dii}) is (up to a constant) the $p$-Laplacian. Then it is clear that $g-u_0$ is the desired solution to problem (\ref{DPNH}).\\
This arguments fails when $p$ is allowed to vary within $\Omega$, for in such case, the integral (\ref{dii}) cannot be expressed in simple terms of the norm (which is known to be uniformly convex if the exponent $p$ is bounded, \cite{Luk})
\vspace{.2in}

In this work, we circumvent this difficulty by utilizing some delicate inequalities (that are new in the literature, to the best of our knowledge) satisfied by the convex modular defined by the functional (\ref{dii}). Our central result is Theorem \ref{main} in which we solve a generalization of problem (\ref{DPNH}).\\

Boundary value problems such as ()\ref{DPNH}) and its homogeneous counterparts constitute the mathematical foundation for the modeling of electrorheological fluids \cite{AM1,AM2, ve-book, RaRu, Ru}. Electrorheological fluids change rapidly and dramatically their viscosity in the presence of a magnetic field; their emerging applications include medicine, civil engineering, military science, among others \cite{BV,ChL,CKKPC,SYCS}.

\vspace{.2in}
	 The structure of the paper is as follows: Section \ref{inequalities} presents a collection of both known and novel fundamental inequalities of Clarkson type. These inequalities play a vital role in the subsequent theory development. Section \ref{modularvectorspaces} provides a concise overview of the modular vector spaces theory. In Section \ref{dirichletintegral} the inequalities established in Section 2 are utilized to handle the minimization of a suitable Dirichlet integral, which in turn is employed in Section \ref{applications} to investigate the solvability of the Dirichlet problem for variable-exponent $p$-Laplacian for non-homogeneous boundary values in $W^{1,p}(\Omega)$.
	
	\section{Auxiliary inequalities}\label{inequalities}
	The subsequent inequalities, inspired by the work of \cite{clarkson}, form the cornerstone of investigating the geometric characteristics of the classical $\ell^p$ and $L^p$ Banach spaces.  These inequalities are poised to assume a pivotal role in the ensuing analysis.
	
	\begin{lemma}\label{p-inequalities}
		For $a,b\in {\mathbb R}$, $|a|+|b|\neq 0$, $1\leq p\leq 2$ (\cite{sundaresan}):
		\begin{equation}\label{scalar-1<p<2} 
			\left|\frac{a+b}{2}\right|^p+\frac{p(p-1)}{2^{p+1}}\frac{|a-b|^2}{(|a|+|b|)^{2-p}}\leq \frac{1}{2}(|a|^p+|b|p).
		\end{equation}
		In addition, if $p\geq 2$ it holds (\cite{clarkson}):
		
		\begin{equation}\label{scalar-p>2}
			\left|\frac{a+b}{2}\right|^p+\left|\frac{a-b}{2}\right|^p\leq \frac{1}{2}(|a|^p+|b|^p).
		\end{equation}
	\end{lemma}
	
	\medskip
	The theory presented in this study necessitates the formulation of vector counterparts to the scalar inequalities mentioned in Lemma \ref{p-inequalities}. Specifically, we seek to establish vector inequalities analogous to (\ref{scalar-1<p<2}) and (\ref{scalar-p>2}), wherein the scalars $a$ and $b$ are substituted with vectors and the absolute values are replaced by a vector space norm $\|\cdot\|$. To achieve the vector-valued version of Lemma \ref{p-inequalities}, we initiate the process with the next technical result, which affirms the validity of these inequalities for complex numbers.\\

	\begin{lemma}\label{complex}
		For $1<p\leq 2$, $z_1\in {\mathbb C}$, $z_2\in {\mathbb C}$, $|z_1|^2+|z_2|^2\neq 0$, it holds
		\begin{equation}\label{complex-1<p<2}
			\left|\frac{z_1+z_2}{2}\right|^{p}+\frac{p(p-1)}{2^{p+1}}\frac{|z_1-z_2|^2}{\left(|z_1|^2+|z_2|^2\right)^{\frac{2-p}{2}}}\leq \frac{1}{2}(|z_1|^p+|z_2|^p).\end{equation}
		In addition, if $p\geq 2$, one has, for any two complex numbers $z_1$ and $z_2$,
		\begin{equation}\label{complex-p>2}
			\left|\frac{z_1+z_2}{2}\right|^{p}+\left|\frac{z_1-z_2}{2}\right|^p\leq \frac{1}{2}(|z_1|^p+|z_2|^p).
		\end{equation}
	\end{lemma}
	\begin{proof}
		Let us first focus on the case $1 < p \leq 2$.  Before, we prove the inequality (\ref{complex-1<p<2}), we will need the following estimate
		\begin{equation}\label{bound-estimate}
			0 \leq (p-1)^{\frac{2}{2-p}}< e^{-2},
		\end{equation}
		for $1\leq p< 2$.   Set $\displaystyle g(p) = (p-1)^{\frac{2}{2-p}}$.  It is easy to show that $g(1)=0$ and that $g(p)\rightarrow e^{-2}$ when $p\rightarrow 2^-.$ Also, $$g^{\prime}(p)=2(p-1)^{\frac{p}{2-p}}(2-p)^{-2}\left(2-p+(p-1)\ln{(p-1)}\right).$$
		Writing $h(p)=2-p+(p-1)\ln{(p-1)}$ it follows that $h(1)=1$, $h(2)=0$ and $h^{\prime}(p)=\ln{(p-1)}<0.$ Thus $h(p)>0$ on $[1,2)$, and hence $g^{\prime}(p)>0$ which gives the estimate.  By setting $w=z_2z_1^{-1}=re^{i\theta}$, $-\pi<\theta\leq \pi$, $r>0$, it is easy to rewrite the target inequality (\ref{complex-1<p<2}) as
		\begin{equation}\label{rewritten} \left|\frac{1+re^{i\theta}}{2}\right|^{p}+\frac{p(p-1)}{2^{p+1}}\frac{|1-re^{i\theta}|^2}{\left(1+r^2\right)^{\frac{2-p}{2}}}\leq \frac{1}{2}(1+r^p).
		\end{equation}
		Fix $r$ and denote the left-hand side by $F(\theta)$, i.e.
		\begin{equation*}
			F(\theta)=\frac{1}{2^p}\left( 1 + 2r\cos{\theta}+r^2\right)^{\frac{p}{2}}+\frac{p(p-1)}{2^{p+1}(1+r^{2})^{\frac{2-p}{2}}} \left(1-2r\cos{\theta}+r^2\right).
		\end{equation*}
		We have
		\begin{equation*}
			F^{\prime}(\theta)=\frac{2pr\sin{\theta}}{2^{p+1}} \left(-(1+2r\cos{\theta}+r^2)^{\frac{p}{2}-1}+\frac{p-1}{(1+r^2)^{\frac{2-p}{2}}}\right).
		\end{equation*}
		It is readily seen that for $\displaystyle -\frac{\pi}{2}\leq \theta\leq\frac{\pi}{2}$, it holds
		\begin{align}\nonumber
			-(1+r^2)^{\frac{2-p}{2}}+(p-1)(1+2r\cos{\theta}+r^2)^{\frac{2-p}{2}}&\leq (p-1)(1+r)^{2-p}-(1+r^2)^{\frac{2-p}{2}}.
		\end{align}
		We claim that
		\begin{equation*}
			(p-1)(1+r)^{2-p}-(1+r^2)^{\frac{2-p}{2}}\leq 0.
		\end{equation*}
		Indeed, it is enough to show that
		\begin{equation*}
			(p-1)^{\frac{2}{p-2}}\leq \frac{1+r^2}{(1+r)^2}.
		\end{equation*}
		This follows directly from the estimate (\ref{bound-estimate}) and the fact that $e^2>2$.  Therefore, $F$ increases on $\displaystyle \left(-\frac{\pi}{2},0\right)$ and decreases on $\displaystyle \left(0, \frac{\pi}{2}\right)$, i.e., on $\displaystyle \left[-\frac{\pi}{2},\frac{\pi}{2}\right]$ one has
		\begin{align}\label{xx}
			F(\theta)\leq F(0)=\left|\frac{1+r}{2}\right|^{p}+\frac{p(p-1)}{2^{p+1}}
			\frac{|1-r|^2}{\left(1+r^2\right)^{\frac{2-p}{2}}}.
		\end{align}
		On the other hand, on $\displaystyle \left(-\pi,-\frac{\pi}{2}\right)\cup \left(\frac{\pi}{2} \pi\right]$, one has $1+2r\cos{\theta}+r^2< 1+r^2$. Consequently,
		\begin{equation*}
			(1+2r\cos{\theta}+r^2)^{\frac{p-2}{2}} > (1+r^2)^{\frac{p-2}{2}}\geq(p-1)(1+r^2)^{\frac{p-2}{2}}.
		\end{equation*}
		Thus, $F(\theta)$ increases on $\displaystyle \left(-\pi,-\frac{\pi}{2}\right)$ and decreases on $\displaystyle \left(\frac{\pi}{2}, \pi\right)$ and the bound in (\ref{xx}) holds on $(-\pi, \pi]$.  On account of inequality (\ref{scalar-1<p<2}), $F(0)$ is bounded above by the right-hand side of inequality (\ref{rewritten}), and this observation proves the desired inequality.\\
		The proof of (\ref{complex-p>2}), for $p > 2$, follows by the same arguments and will be omitted.
	\end{proof}
	
	With the assistance of the aforementioned lemma,we now tackle the vector form of the inequalities derived from Lemma \ref{p-inequalities}, within any Hilbert space.\\
	
	\begin{theorem} \label{vector-p-inequalities}  Let ${\mathbf u}$, ${\mathbf v}$ be vectors in a Hilbert space $({\mathbb H}, \|\cdot\|)$.  If $1\leq p\leq 2$ it holds
		\begin{equation}\label{1<p<2-vector}
			\left\|\frac{\mathbf{u} + \mathbf{v}}{2}\right\|^p + \frac{p(p-1)}{2^{p+1}}\frac{\|\mathbf{u}-\mathbf{v}\|^2}{(\|\mathbf{u}\|+\|\mathbf{v}\|)^{2-p}} \leq \frac{1}{2}(\|\mathbf{u}\|^p+\|\mathbf{v}\|^p),
		\end{equation}
		provided $\|\mathbf u\| + \|\mathbf v\| \neq 0$.  In addition, if $p\geq 2$ it holds
		\begin{equation}\label{p>2-vector}
			\left\|\frac{\mathbf{u} + \mathbf{v}}{2}\right\|^p + \left\|\frac{\mathbf{u} - \mathbf{v}}{2}\right\|^p \leq \frac{1}{2}(\|\mathbf{u}\|^p+\|\mathbf{v}\|^p).
		\end{equation}
	\end{theorem}
	\begin{proof}  If the vectors ${\mathbf u}$, ${\mathbf v}$ are linearly dependent, the two inequalities reduce to the scalar case.  Assume that ${\mathbf u}$ and ${\mathbf v}$ are linearly independent.  Set $W$ the subspace of ${\mathbb H}$ spanned by these two vectors.  Using Gram-Schmidt, there exists an orthonormal basis $\{\mathbf{I}, \mathbf{J}\}$ of $W$.  We have
		$$\mathbf{u} = x \mathbf{I} + y \mathbf{J}\;\;\; and\;\;\; \mathbf{v} = a \mathbf{I} + b \mathbf{J},$$
		for $(x,y) \in \mathbb{R}^2$ and $(a,b) \in \mathbb{R}^2$.  Set $z_1 = x+ i y$ and $z_2 = a + i b$ in $\mathbb C$.
		Clearly the following hold
		$$\left\{\begin{array}{clll}
			\|\mathbf{u}\|^2 &= |z_1|^2 = x^2 + y^2 ,\\
			\|\mathbf{v}\|^2 &= |z_2|^2 = a^2 + b^2,\\
			\|\mathbf{u}+ \mathbf{v}\|^2 &= |z_1 + z_2 |^2 = (x+a)^2 + (y+b)^2,\\
			\|\mathbf{u}- \mathbf{v}\|^2 &= |z_1 - z_2 |^2 = (x-a)^2 + (y-b)^2.
		\end{array}\right.$$
		Lemma \ref{complex} implies
		$$\left|\frac{z_1+z_2}{2}\right|^{p}+\frac{p(p-1)}{2^{p+1}}\frac{|z_1-z_2|^2}{\left(|z_1|^2+|z_2|^2\right)^{\frac{2-p}{2}}}\leq \frac{1}{2}(|z_1|^p+|z_2|^p),$$
		for $1<p\leq 2$, and for $p\geq 2$, we have
		$$\left|\frac{z_1+z_2}{2}\right|^{p}+\left|\frac{z_1-z_2}{2}\right|^p\leq \frac{1}{2}(|z_1|^p+|z_2|^p),$$
		which obviously imply
		$$\left\|\frac{\mathbf{u} + \mathbf{v}}{2}\right\|^p + \frac{p(p-1)}{2^{p+1}}\frac{\|\mathbf{u}-\mathbf{v}\|^2}{(\|\mathbf{u}\|+\|\mathbf{v}\|)^{2-p}} \leq \frac{1}{2}(\|\mathbf{u}\|^p+\|\mathbf{v}\|^p),$$
		for $1 < p \leq 2$, provided $\|\mathbf u\| + \|\mathbf v\| \neq 0$, and if $p\geq 2$ it holds
		$$\left\|\frac{\mathbf{u} + \mathbf{v}}{2}\right\|^p + \left\|\frac{\mathbf{u} - \mathbf{v}}{2}\right\|^p \leq \frac{1}{2}(\|\mathbf{u}\|^p+\|\mathbf{v}\|^p).$$
		The proof of Theorem \ref{vector-p-inequalities} is complete.
	\end{proof}
	\begin{remark}\label{palchi}{\normalfont The preceding inequalities will be used (Theorem \ref{UM}) in the particular case of vectors ${\mathbf u}=(u_1,...,u_n)$ and ${\mathbf v}=(v_1,...,v_n)$ in ${\mathbb R}^n$ with Euclidean norms $|{\mathbf u}|=\left(\sum\limits_{1}^n|u_j|^2\right)^{\frac{1}{2}}$, $|{\mathbf v}|=\left(\sum\limits_{1}^n|v_j|^2\right)^{\frac{1}{2}}$.
}
\end{remark}
	\section{Modular vector spaces, variable exponent spaces}\label{modularvectorspaces}
	Since the variable exponent $p$-Laplacian $\Delta_p$ is the Fr\'{e}chet derivative of the Dirichlet integral (which is modular in nature) the norm structure of the variable exponent Sobolev space $W^{1,p}(\Omega)$ is insufficient for its study.\\	 
	 Aiming at demonstrating the importance of the modular structure of $W^{1,p}(\Omega)$ we present a concise overview of the theory of modular spaces. For a more comprehensive exploration of the subjects merely touched upon in this section, interested readers are directed to \cite{ve-book, KR, ML, musielak_book}.\\
	
	\begin{definition}\label{def-modular}\cite{musielak_book, nakano}
		A convex modular on a real vector space $X$ is a function $\varrho: X \to [0,\infty]$ satisfying the following conditions:
		\begin{enumerate}
			\item[(1)] $\varrho(x) = 0$ if and only if $x = 0$,
			\item[(2)] $\varrho(\alpha x) = \varrho(x)$, if $|\alpha| =1$,
			\item[(3)] $\varrho(\alpha x + (1-\alpha) y )\leq \alpha\varrho(x) + (1-\alpha)\varrho(y)$, for any $\alpha \in [0,1]$
			and any $x, y \in X$.
		\end{enumerate}
		In addition, $\varrho$ is said to be left-continuous
		if $\lim\limits_{r \to 1-} \ \varrho(r x) = \varrho(x)$ for any $x \in X$.  \\
	\end{definition}
	
	\medskip
	\noindent A modular function on a vector space $X$ naturally gives rise to a modular space.
	
	\begin{definition}\label{modular-space}  Given a convex modular $\varrho$ defined on the vector space $X$, the modular space generated by $\varrho$ is the set
		$$X_\varrho = \{x \in X; \ \lim\limits_{\alpha \to 0} \ \varrho(\alpha x) = 0 \}.$$
		The Luxemburg norm on $X$,  $\|.\|_\varrho: X_\varrho \to [0,\infty)$, is defined by
		$$ \|x\|_\varrho := \inf \left\{\alpha > 0;\; \varrho\left(\frac{x}{\alpha}\right) \leq 1\right\}.$$
	\end{definition}
	
	\medskip
	Variable exponent Lebesgue spaces, initially introduced in 1931 by Orlicz \cite{orlicz1931}, have gained substantial attention in recent years. For a comprehensive exploration of these spaces, including in-depth analysis and discussions, interested readers are encouraged to refer to \cite{ve-book, KR, ML}.
	
	\medskip
	\begin{definition}\label{ves-continuous}  Consider a domain $\Omega\subset {\mathbb R}^n$. We denote the vector space of all real-valued, Borel-measurable functions defined on $\Omega$ as ${\mathcal M}(\Omega)$. Within ${\mathcal M}(\Omega)$, we define ${\mathcal P}(\Omega)$ as the subset consisting of functions $p:\Omega\longrightarrow [1,\infty]$. For each function $p$ in ${\mathcal P}(\Omega)$, we define the set $\Omega_{\infty}$ as follows: $\Omega_{\infty} := \left\{  x\in\Omega:p(x)=\infty\right\}$.\\
		Now, we introduce the function $\varrho_p:{\mathcal M}(\Omega)\longrightarrow [0,\infty]$ defined as follows:
		$$\varrho_p(u)=\int\limits_{\Omega \setminus \Omega_{\infty}}|u(x)|^{p(x)}d\mu +\underset{x\in\Omega_{\infty}} \sup|u(x)|.$$
		This function $\varrho_p(u)$ is a convex and continuous modular on ${\mathcal M}(\Omega)$. We refer to the associated modular vector space as $L^{p(\cdot)}(\Omega)$ or simply $L^p(\Omega)$ when there is no potential for confusion.
	\end{definition}
	\medskip
	
	The associated Sobolev space is defined in the following manner:
	
	\begin{definition}
		For $\Omega$ and $p$ as in the preceding definition, $W^{1,p}(\Omega)$ will stand for the vector subspace of $L^p(\Omega)$ consisting of those functions whose weak derivatives also belong to $L^p(\Omega)$.
	\end{definition}
	A technical point arises when it comes to defining a modular on $W^{1,p}(\Omega)$, for given any norm in ${\mathbb R}^n$, say $|\cdot|$, then the functional
	\begin{align}
		\rho_{|\cdot|}:W^{1,p}(\Omega)\rightarrow [0,\infty]\\ \nonumber
		\rho_{|\cdot|}(u)=\varrho (u)+\varrho\left(|\nabla u|\right)
	\end{align}
	is a convex modular. In the sequel only the Euclidean norm in ${\mathbb R}^n$ will be considered and the corresponding modular will be simply denoted by $\rho$. Hence, the Sobolev space $W^{1,p}(\Omega)$ will be endowed with the modular
	\begin{align}\label{modularw1p}
		\rho_{|\cdot|}&:W^{1,p}(\Omega)\rightarrow [0,\infty]\\ \nonumber
		\rho_{|\cdot|}(u)&=\varrho (u)+\varrho\left(\left(\sum_{j=1}^n\left(\frac{\partial u}{\partial x_j}\right)^2\right)^{\frac{1}{2}}\right)
	\end{align}
	
	\section{The Dirichlet energy integral}\label{dirichletintegral}
	The primary objective of this section is to utilize the previously established results and techniques to examine the minimization of specific integrals related to Dirichlet energy.
	
	\medskip
	Throughout the subsequent discussion, let $\Omega \subset {\mathbb R}^n$ denote a bounded domain with a smooth ($C^{2}$) boundary. We will utilize $W^{1,p}_0(\Omega)$ to represent the closure, in the Luxemburg norm, of $W^{1,p}(\Omega)$, specifically, the closure of $C_0^{\infty}(\Omega)$. Furthermore, $\left(W^{1,p}_0(\Omega)\right)^{\ast}$ will refer to the topological dual of $W^{1,p}_0(\Omega)$. It is important to note that the Luxemburg norm considered here corresponds to the modular (\ref{modularw1p}).\\
	
	Throughout, we use the notations
	$$\rho_{\nabla, \Omega^*}(u) = \int\limits_{\Omega^*}\frac{|\nabla u|^p}{p}dx\;\; and\;\;\; \rho_{\Omega^*}(u)= \int\limits_{\Omega^*}\frac{ |u|^p}{p}dx,$$
	where $\Omega^*$ is a subset of $\Omega$.  When $\Omega^* = \Omega$, we write $\rho_{\nabla, \Omega} = \rho_\nabla$ and $\rho_{\Omega} = \rho$.  Note that we have $\rho_\nabla(u) = \rho(|\nabla u|)$.  Moreover the functional $\rho$ is a convex modular on $L^p(\Omega)$; the associated Luxemburg norm
	\begin{equation}
		\|u\|_p=\inf\left\{\lambda >0: \rho\left(\frac{u}{\lambda}\right)\leq 1\right\}
	\end{equation}
	is subject to the inequality $\|u\|^{p_{-}}_p\leq \rho(u)$ whenever $\|u\|_p\geq 1$.
	
	\begin{theorem} \label{boundedbelow}
		Let $\Omega \subset {\mathbb R}^n$ be a smooth, bounded domain ($C^{1,\alpha}$ will suffice), $p\in C(\overline{\Omega})$, $p_->1$.  Consider $f\in \left(W_0^{1,p}(\Omega)\right)^{\ast}$. Then, for any $\varphi \in W^{1,p}(\Omega)$, the functional $F: W_0^{1,p}(\Omega)\rightarrow {\mathbb R}$ defined by
		\begin{align}
			F(u)= \rho_\nabla(\varphi -u) - f(u),
		\end{align}
		is bounded below.
	\end{theorem}
	\begin{proof}
		Assume $\||\nabla (\varphi-u)|\|_p\geq 1$.  Then it holds
		$$\begin{array}{lll}
			F(u)&= \rho_\nabla(\varphi -u) - f(u)\\
			&\geq \rho_\nabla(\varphi -u) - \|f\|_{\left(W_0^{1,p}(\Omega)\right)^{\ast}}\|u\|_{W_0^{1,p}(\Omega)}\\
			&= \rho_\nabla(\varphi -u) - \|f\|_{\left(W_0^{1,p}(\Omega)\right)^{\ast}}\||\nabla u|\|_{L^{p}(\Omega)}\\
			&\geq \rho_\nabla(\varphi -u) - \|f\|_{\left(W_0^{1,p}(\Omega)\right)^{\ast}}(\||\nabla (u-\varphi)|\|_{L^p(\Omega)}+\||\nabla \varphi|\|_{L^p(\Omega)})\\
			&\geq \Big(\rho_\nabla(\varphi -u)\Big)^{\frac{1}{p_-}}\left(\Big(\rho_\nabla(\varphi -u)\Big)^{1-\frac{1}{p_-}}-\|f\|_{\left(W_0^{1,p}(\Omega)\right)^{\ast}}\right) -\||\nabla \varphi|\|_{L^p(\Omega)}.
		\end{array}$$
		If we set
		$$h(x) = x^{\frac{1}{p_-}}\left(x^{1-\frac{1}{p_-}} - \|f\|_{\left(W_0^{1,p}(\Omega)\right)^{\ast}}\right) - \||\nabla \varphi|\|_{L^p(\Omega)}.$$
		A straightforward calculation shows that $h(x)$ is bounded below for $x \in [1,+\infty)$.
		On the other hand, if $\||\nabla (\varphi-u)|\|_p\leq 1$, it follows
		$$\begin{array}{ll}
			F(u)&\geq \rho_\nabla(\varphi -u)-\|f\|_{\left(W_0^{1,p}(\Omega)\right)^{\ast}}\left(1+\||\nabla \varphi|\|_{L^p(\Omega)}\right) \\
			&\geq -\|f\|_{\left(W_0^{1,p}(\Omega)\right)^{\ast}}\left(1+\||\nabla \varphi|\|_{L^p(\Omega)}\right).
		\end{array}$$
		Thus, $F$ is bounded below as claimed.
	\end{proof}
	
	\medskip
	In the following theorem, we discuss the minimization of a variation of the Dirichlet energy integral.
	
	\medskip
	\begin{theorem}\label{UM}
		Let $\Omega \subset {\mathbb R}^n$ be a smooth, bounded domain ($C^{1,\alpha}$ will suffice), $p\in C(\overline{\Omega})$, $p_->1$.  Let $q:\Omega\rightarrow [0,\infty)$ be a non-negative, measurable function.
		Consider $f\in \left(W_0^{1,p}(\Omega)\right)^{\ast}$.  Consider the functional $G: W_0^{1,p}(\Omega)\rightarrow {\mathbb R}$ defined by
		$$G(u)= \rho_\nabla(u-\varphi) + \int\limits_{\Omega}\frac{q}{p}|u-\varphi|^p\,dx - f(u),$$
		where $\varphi \in W^{1,p}(\Omega)$.
		Then any minimizing sequence $(u_n)$ of $G$ is convergent.  Its limit is independent of the minimizing sequence and is the unique minimizer of $G$.
	\end{theorem}
	\begin{proof}
		Throughout the proof, we will the notation
		$$\rho_q(u) = \int\limits_{\Omega}\frac{q}{p}|u|^p\,dx.$$
		Note that $G(u) = F(u) +\rho_q(u-\varphi) \geq F(u)$, for any $u \in W_0^{1,p}(\Omega)$.  Theorem \ref{boundedbelow} will then imply that $G$ is bounded below.  Set
		$$d=\inf_{v\in W_0^{1,p}(\Omega)}\ G(v) > -\infty.$$
		Let $(u_n)\subset W_0^{1,p}(\Omega) $ be a minimizing sequence of $G$, i.e.,  $\lim\limits_{n \to \infty} G(u_n) = d$.  Set $\eta_n = G(u_n) -d \geq 0$, for any $n \in \mathbb{N}$.  Clearly we have $\lim\limits_{n \to \infty} \eta_n =0$.  First note that $(u_n)$ is bounded in $W^{1,p}_0(\Omega)$.  Assume not.  Without loss of any generality, we assume that $\|\nabla (u_n-\varphi)\|_{p}\rightarrow \infty$.  Then $\|\nabla (u_n- \varphi)\|_{p}>1$ for $n\geq N$, for some $N \in \mathbb{N}$ and thus, the inequality
		$$\frac{1}{p_+}\|\nabla (u_n-\varphi)\|_{p}^{p_-}\leq  \rho_\nabla(\varphi -u_n) \leq G(u_n) + f(u_n),$$
		would imply
		\begin{align*}
			\frac{1}{p_+}\|\nabla (u_n-\varphi)\|_{p}^{p_-}
			&\leq \|f\|_{\left(W^{1,p}_0(\Omega)\right)^{\ast}}\|\nabla u_n\|_{p}+ d + \sup\limits_{k \in \mathbb{N}} \eta_k \\ \nonumber
			&\leq \|f\|_{\left(W^{1,p}_0(\Omega)\right)^{\ast}} \|\nabla (u_n-\varphi)\|_{p} \\ &\mbox{\hspace*{2cm}} +\|f\|_{\left(W^{1,p}_0(\Omega)\right)^{\ast}}\|\nabla \varphi\|_{p}+ d + \sup\limits_{k \in \mathbb{N}} \eta_k,
		\end{align*}
		which is certainly not possible since $p_->1$.  Therefore, $(u_n)$ is bounded in $W^{1,p}_0(\Omega)$ as claimed. Poincar\'{e}'s inequality yields then the boundedness of $(u_n)$ in $W^{1,p}_0(\Omega)$.  Next, we prove that the minimizing sequence $(u_n)$ is Cauchy in $W_0^{1,p}(\Omega)$.   Assume not.  Then there exists $\varepsilon_0 > 0$ such that for any $I \in \mathbb{N}$, there exists $j,k > I$ such that
		\begin{equation}\label{notcauchy}
			\rho_\nabla\left(\frac{u_{j}-u_{k}}{2}\right) \geq \varepsilon_0.
		\end{equation}
		Using the boundedness of $(u_n)$, we set $\displaystyle C = \sup_{n \in \mathbb{N}} \ \rho_\nabla(u_n) < +\infty$.  Note that $C > 0$.  Set
		$$\gamma = \min\left\{\frac{\varepsilon_0}{8(p_{-} -1)C}, \frac{\varepsilon_0}{4\ C}, \frac{1}{2}\right\}.$$
		Then
		$$\frac{\varepsilon_0}{2} - \frac{(p_{-} -1)\ \gamma}{2}\Big(\rho_\nabla(u_m) + \rho_\nabla(u_k) \Big)
		\geq \frac{\varepsilon_0}{2} - (p_{-} -1)\ \gamma\ C > \frac{\varepsilon_0}{4},$$
		for any $m, k \in \mathbb{N}$.  Let $I_0 \in \mathbb{N}$ such that for any $n \geq I_0$, we have
		$$\eta_n < \min\left\{\frac{\varepsilon_0}{8}, \frac{\gamma (p_{-} -1)\varepsilon_0}{16}\right\}.$$
		Set $\Omega_1=\{x\in \Omega: p(x)\geq 2\}$.  Fix $m, k > I_0$ such that (\ref{notcauchy}) is satisfied.  On account of (\ref{notcauchy}), we have $\displaystyle \rho_{\nabla, \Omega_1} \left(\frac{u_m-u_k}{2}\right)\geq \frac{\varepsilon_0}{2}$ or $\displaystyle \rho_{\nabla, \Omega\setminus \Omega_1} \left(\frac{u_m-u_k}{2}\right)\geq \frac{\varepsilon_0}{2}$. Assume
		$$\rho_{\nabla, \Omega_1} \left(\frac{u_m-u_k}{2}\right) \geq \frac{\varepsilon_0}{2}$$
		holds.  Then by virtue of Remark \ref{palchi} and the first inequality in Theorem \ref{vector-p-inequalities} , it would then follow that
		$$\begin{array}{lll}
			\displaystyle \rho_{\nabla, \Omega_1} \left(\frac{u_m+u_k}{2} - \varphi\right)+ \frac{\varepsilon_0}{2}&\leq \displaystyle \rho_{\nabla, \Omega_1} \left(\frac{u_m+u_k}{2} - \varphi\right) + \rho_{\nabla, \Omega_1} \left(\frac{u_m-u_k}{2}\right) \\
			&\leq \displaystyle  \frac{1}{2}\rho_{\nabla, \Omega_1} \left(u_k - \varphi\right)
			+\frac{1}{2}\rho_{\nabla, \Omega_1} \left(u_m - \varphi\right).
		\end{array}$$
		The convexity of $W_0^{1,p}(\Omega)$ yields
		$$\begin{array}{lll}
			d&\displaystyle \leq \rho_\nabla\left(\frac{u_k+u_m}{2}-\varphi\right) + \rho_q\left(\frac{u_k+u_m}{2}-\varphi\right)- f\left(\frac{u_k+u_m}{2}\right)\\
			&\displaystyle = \Big(\rho_{\nabla, \Omega_1} + \rho_{\nabla, \Omega\setminus \Omega_1}\Big) \left(\frac{u_k+u_m}{2}-\varphi\right) -f\left(\frac{u_k+u_m}{2}\right) + \rho_q\left(\frac{u_k+u_m}{2}-\varphi\right)\\
			&\leq \displaystyle \rho_{\nabla, \Omega_1}\left(\frac{u_k+u_m}{2}-\varphi\right) + \frac{1}{2}\rho_{\nabla, \Omega\setminus \Omega_1}(u_k-\varphi) + \frac{1}{2}\rho_{\nabla, \Omega\setminus \Omega_1}(u_m-\varphi)\\
			&\displaystyle \mbox{\hspace*{2cm}} -f\left(\frac{u_k+u_m}{2}\right) + \rho_q\left(\frac{u_k+u_m}{2}-\varphi\right) . \\
		\end{array}$$
		The previous estimate yields
		$$\begin{array}{lll}
			d&\leq \displaystyle  \frac{1}{2}\rho_{\nabla, \Omega_1} \left(u_k - \varphi\right)
			+\frac{1}{2}\rho_{\nabla, \Omega_1} \left(u_m - \varphi\right) - \frac{\varepsilon_0}{2} + \frac{1}{2}\rho_{\nabla, \Omega\setminus \Omega_1}(u_k-\varphi)\\
			&\displaystyle \mbox{\hspace*{2cm}} + \frac{1}{2}\rho_{\nabla, \Omega\setminus \Omega_1}(u_m-\varphi) -f\left(\frac{u_k+u_m}{2}\right) + \rho_q\left(\frac{u_k+u_m}{2}-\varphi\right) \\
			&= \displaystyle \frac{1}{2}\rho_\nabla \left(u_k - \varphi\right)
			+\frac{1}{2}\rho_\nabla \left(u_m - \varphi\right) - \frac{\varepsilon_0}{2}-f\left(\frac{u_k+u_m}{2}\right) + \rho_q\left(\frac{u_k+u_m}{2}-\varphi\right) \\
		\end{array}$$
		which implies
		$$d \leq \frac{1}{2}\Big(\rho_\nabla \left(u_k - \varphi\right) + \rho_q\left(u_k-\varphi\right)\Big) + \frac{1}{2}\Big(\rho_\nabla \left(u_m - \varphi\right) + \rho_q\left(u_m-\varphi\right) -f(u_m)\Big) - \frac{\varepsilon_0}{2}.$$
		Using the definition of $G$ it follows that
		$$d \leq \frac{1}{2} G(u_k) + \frac{1}{2} G(u_m) -\frac{\varepsilon_0}{2} \leq d+ \frac{\eta_m + \eta_k}{2}-\frac{\varepsilon_0}{2}\leq d+\frac{\varepsilon_0}{8} -\frac{\varepsilon_0}{2} = d - \frac{3 \varepsilon_0}{8},$$
		which is a contradiction.  Next, we assume
		$$\rho_{\nabla, \Omega\setminus \Omega_1} \left(\frac{u_k-u_m}{2}\right) \geq \frac{\varepsilon_0}{2}.$$
		Set $\Omega_2=\{x\in \Omega\setminus \Omega_1: |\nabla (u_k-u_m)|\leq \gamma (|\nabla u_k|+|\nabla u_m|)\}$.
		We claim that the following inequality holds
		\begin{equation}\label{differenceoncomplement}
			\rho_{\nabla, \Omega\setminus (\Omega_1\cup \Omega_2)} \left(\frac{u_k-u_m}{2}\right) \geq \frac{\varepsilon_0}{4}.
		\end{equation}
		Indeed, observe first that
		$$\rho_{\nabla, \Omega_2} \left(\frac{u_k-u_m}{2}\right) \leq
		\frac{\gamma}{2}\int\limits_{\Omega_2}\frac{1}{p}\left(|\nabla u_k|^p+|\nabla u_m|^p\right),$$
		from which it follows that
		\begin{align*}
			\rho_{\nabla, \Omega\setminus (\Omega_1\cup \Omega_2)} \left(\frac{u_k-u_m}{2}\right) & =
			\rho_{\nabla, \Omega\setminus \Omega_1} \left(\frac{u_k-u_m}{2}\right) - \rho_{\nabla, \Omega_2} \left(\frac{u_k-u_m}{2}\right)\\
			&\geq \frac{\varepsilon_0}{2}-\frac{\gamma}{2}\int\limits_{\Omega_2}\frac{1}{p}\left(|\nabla u_k|^p+|\nabla u_m|^p\right)\\
			& \geq \frac{\varepsilon_0}{2} - \gamma \ C \geq \frac{\varepsilon_0}{4},
		\end{align*}
		as claimed.  Observe that $|\nabla (u_m-u_k)|>\gamma \Big(|\nabla u_k|+|\nabla u_m|\Big)$ on $\Omega\setminus\left(\Omega_1\cup \Omega_2\right)$, and therefore the following hold
		\begin{align*}
			&\frac{1}{p}\left|\nabla\left(\frac{u_k+u_m}{2}-\varphi\right)\right|^pdx+
			\frac{\gamma^{2-p}(p-1)}{2}\left |\frac{\nabla (u_k-u_m)}{2}\right|^p \\ \nonumber &\leq \frac{1}{p}\left|\nabla\left(\frac{u_k+u_m}{2}-\varphi\right)\right|^pdx \\ \nonumber &+
			\frac{p(p-1)}{2p}\left|\frac{\nabla (u_k-u_m)}{|\nabla u_k|+|\nabla u_m|}\right|^{2-p}\left|\frac{\nabla (u_k-u_m)}{2}\right|^p\\ \nonumber &\leq \frac{1}{2p}\left( |\nabla (\varphi-u_k)|^p+|\nabla (\varphi-u_m)|^p \right),
		\end{align*}
		by virtue of the second inequality in Theorem \ref{vector-p-inequalities}. From the preceding estimate it is clear that
		\begin{align*}
			&\rho_{\nabla, \Omega\setminus (\Omega_1\cup \Omega_2)} \left(\frac{u_k+u_m}{2}-\varphi\right) + \frac{\gamma(p_--1)}{2}\frac{\varepsilon_0}{4}\\
			&\leq \rho_{\nabla, \Omega\setminus (\Omega_1\cup \Omega_2)} \left(\frac{u_k+u_m}{2}-\varphi\right)
			+\int\limits_{\Omega \setminus \left(\Omega_1\cup\Omega_2\right)}\frac{\gamma\ p\ (p-1)}{2} \frac{1}{p}\left |\frac{\nabla (u_k-u_m)}{2}\right|^pdx \\ \nonumber
			&\leq \rho_{\nabla, \Omega\setminus (\Omega_1\cup \Omega_2)} \left(\frac{u_k+u_m}{2}-\varphi\right)
			+\int\limits_{\Omega \setminus \left(\Omega_1\cup\Omega_2\right)}\frac{\gamma^{2-p}(p-1)}{2}\left |\frac{\nabla (u_k-u_m)}{2}\right|^pdx \\
			& \leq \frac{1}{2}\Big(\rho_{\nabla, \Omega\setminus (\Omega_1\cup \Omega_2)} \left(u_k-\varphi\right) + \rho_{\nabla, \Omega\setminus (\Omega_1\cup \Omega_2)} \left(u_m-\varphi\right)\Big)
		\end{align*}
		which implies
		$$\begin{array}{lll}
			\rho_{\nabla, \Omega\setminus (\Omega_1\cup \Omega_2)} \left(\frac{u_k+u_m}{2}-\varphi\right) &\leq
			\displaystyle \frac{1}{2}\Big(\rho_{\nabla, \Omega\setminus (\Omega_1\cup \Omega_2)} \left(u_k-\varphi\right) + \rho_{\nabla, \Omega\setminus (\Omega_1\cup \Omega_2)} \left(u_m-\varphi\right)\Big) \\
			&\displaystyle \mbox{\hspace*{3cm}} - \frac{\gamma\ (p_--1)\ \varepsilon_0}{8}.
		\end{array}$$
		Finally,
		\begin{align}\label{finally} \nonumber
			\rho_\nabla\left(\frac{u_k+u_m}{2}-\varphi\right) &= \rho_{\nabla, \Omega_1\cup \Omega_2}\left(\frac{u_k+u_m}{2}-\varphi\right) + \rho_{\nabla, \Omega\setminus (\Omega_1\cup \Omega_2} \left(\frac{u_k+u_m}{2}-\varphi\right)\\ \nonumber
			&\leq \frac{1}{2} \rho_{\nabla, \Omega_1\cup \Omega_2}(u_k - \varphi) + \frac{1}{2} \rho_{\nabla, \Omega_1\cup \Omega_2}(u_m - \varphi) \\ \nonumber
			&\mbox{\hspace*{3cm}}+ \rho_{\nabla, \Omega\setminus (\Omega_1\cup \Omega_2)} \left(\frac{u_k+u_m}{2}-\varphi\right)\\ \nonumber
			&\leq \frac{1}{2} \rho_{\nabla, \Omega_1\cup \Omega_2}(u_k - \varphi) + \frac{1}{2} \rho_{\nabla, \Omega_1\cup \Omega_2}(u_m - \varphi) \\ \nonumber
			&\mbox{\hspace*{1cm}}+\frac{1}{2}\rho_{\nabla, \Omega\setminus (\Omega_1\cup \Omega_2)} \left(u_k-\varphi\right)+\frac{1}{2} \rho_{\nabla, \Omega\setminus (\Omega_1\cup \Omega_2)} \left(u_m-\varphi\right)\\
			&\mbox{\hspace*{3cm}}- \frac{\gamma\ (p_--1)\ \varepsilon_0}{8}.
		\end{align}
		Hence
		$$\begin{array}{lll}
			\displaystyle \rho_\nabla\left(\frac{u_k+u_m}{2}-\varphi\right) &\displaystyle + \rho_q \left(\frac{u_k+u_m}{2}-\varphi\right) - f\left(\frac{u_k+u_m}{2}\right) \leq\\
			&\displaystyle \frac{1}{2}\Big(\rho_\nabla(u_k-\varphi) + \rho_\nabla(u_m-\varphi)\Big)- \frac{\gamma\ (p_--1)\ \varepsilon_0}{8}\\
			&\mbox{\hspace*{1cm}} \displaystyle + \rho_q \left(\frac{u_k+u_m}{2}-\varphi\right) - f\left(\frac{u_k+u_m}{2}\right)\\
			&\displaystyle \leq \frac{1}{2}\Big(\rho_\nabla(u_k-\varphi) + \rho_\nabla(u_m-\varphi)\Big)- \frac{\gamma\ (p_--1)\ \varepsilon_0}{8}\\
			&\mbox{\hspace*{1cm}} \displaystyle + \frac{1}{2}\Big(\rho_q (u_k-\varphi)+\rho_q (u_m-\varphi)\Big) - \frac{f(u_k)+f(u_m)}{2}\\
			&= \displaystyle \frac{1}{2}\rho_\nabla(u_k-\varphi)+\frac{1}{2}\rho_q (u_k-\varphi)-\frac{1}{2}f(u_k) +\frac{1}{2}\rho_\nabla(u_m-\varphi)\\
			&\displaystyle \mbox{\hspace*{1cm}} +\frac{1}{2}\rho_q (u_m-\varphi)-\frac{1}{2}f(u_m) - \frac{\gamma\ (p_--1)\ \varepsilon_0}{8}.
		\end{array}$$
		Using the definition of $G$, it is concluded that
		\begin{align*}
			d\leq G\left(\frac{u_k+u_m}{2}\right) &\leq \frac{G(u_k) + G(u_m)}{2}-\frac{\gamma(p_--1)\varepsilon_0}{8}  \\
			&\leq d+ \frac{\eta_m + \eta_k}{2}-\frac{\gamma(p_--1)\varepsilon_0}{8}\\
			&\leq d+\frac{\gamma(p_--1)\varepsilon_0}{16} -\frac{\gamma(p_--1)\varepsilon_0}{8} \\
			& = d -\frac{\gamma(p_--1)\varepsilon_0}{16},
		\end{align*}
		which is a contradiction.  Thus $(u_j)$ is Cauchy in $W_0^{1,p}(\Omega)$ as claimed.  Clearly its limit is a minimizer of $G$. The uniqueness follows easily since the minimizing sequence was chosen arbitrarily.
	\end{proof}
	
	\medspace
	In the next section the preceding minimization result is applied to deal with the solvability of boundary value problems.
	
	\section{Applications to partial differential equations}\label{applications}
	The focus of this section is to apply the functional-analytic tools and techniques that have been developed thus far to investigate the solvability of a family of boundary value problems.\\The main result in this section is the following theorem.
	
	\medskip
	
	\begin{theorem}\label{main}
		Let $\Omega\subset {\mathbb R}^n$ be a bounded domain with  smooth boundary $\partial\Omega$, $p\in C(\overline{\Omega})$, $p_->1.$ Then, for any $\varphi \in W^{1,p}(\Omega)$, any $0\leq q\in C(\overline{\Omega})$ and any $f\in \left( W_0^{1,p}(\Omega)\right)^{\ast}$, there exists a unique solution $u\in W^{1,p}(\Omega)$ to the boundary value problem
		\begin{equation}\label{general}
			\begin{cases}
				-\Delta_p(u)+q|u|^{p-2}u=f \,\,\,\text{in}\,\,\, \Omega\\
				u|_{\partial \Omega}= \varphi.
			\end{cases}
		\end{equation}
		
	\end{theorem}
	\begin{proof}
		The proof follows by observing that the differential operator in Problem (\ref{general}) is the Fr\'{e}chet derivative of the functional
		\begin{align}
			&G:W_0^{1,p}(\Omega)\rightarrow [0,\infty)\\ \nonumber
			&G(u)=\int\limits_{\Omega}\frac{1}{p}|\nabla u|^p\,dx+\int\limits_{\Omega}\frac{q}{p}|u|^p\,dx-f(u),
		\end{align}
		introduced in Theorem \ref{UM}. Let $v$ be the unique minimizer of $G$ whose existence and uniqueness follows from Theorem \ref{UM}. Then the function $\varphi-v$ is the sought-for solution of Problem (\ref{general}).\\
		To proceed with uniqueness observe first that for any vectors $u,v\in {\mathbb R}^n$ it holds the identity
		\begin{equation}
			(|u|^{p-2}u-|v|^{p-2}v)(u-v)=\frac{(|u|^{p-2}-|v|^{p-2})(|u|^2-|v|^2)}{2}+\frac{1}{2}(|u|^{p-2}+|v|^{p-2})|u-v|^2,
		\end{equation}
		from which the inequality (here $\gamma(p)=2$ if $2\leq p<3$ and $\gamma(p)=2^{2-p}$ if $p\geq 3$):
		\begin{equation}\label{p>2}
			|u-v|^p\leq \gamma(p)(|u|^{p-2}u-|v|^{p-2}v)(u-v)
		\end{equation}
		is immediate. \\
		Likewise, for $1<p\leq 2$ it holds
		\begin{equation}\label{p<2}
			(p-1)|u-v|^2\left (1+|u|^2+|v|^2\right)^{\frac{p-2}{2}}\leq (|u|^{p-2}u-|v|^{p-2}v)(u-v).
		\end{equation}
		\medskip
		If $u_0\in W^{1,p}(\Omega)$ and $u_1\in W^{1,p}(\Omega)$ are solutions of the Dirichlet problem (\ref{DP}), then by definition, for any $h\in W^{1,p}_0$ one has
		\begin{equation}\label{int}
			I_h=\int\limits_{\Omega}\frac{1}{p}\left(|\nabla u_0|^{p-2}\nabla u_0 -|\nabla u_1|^{p-2}\nabla u_1\right)\nabla h\,dx+\int_{\Omega}\frac{q}{p}\left(|u_0|^{p-2}u_0-|u_1|^{p-2}u_1\right)h=0.
		\end{equation}
		In particular, the preceding inequality holds for $h=u_0-u_1\in W_0^{1,p}(\Omega)$, on account of the boundary condition. Notice that by virtue of inequalities (\ref{p>2}) and (\ref{p<2}), one has
		\begin{equation}
			\int\limits_{\Omega}\frac{q}{p}\left(|u_0|^{p-2}u_0-|u_1|^{p-2}u_1\right)(u_0-u_1)dx\geq 0.
		\end{equation}
		Thus,
		\begin{align}\nonumber
			0=I_{u_0-v_0}&\geq \left(\int\limits_{1<p<2}+\int\limits_{p\geq 2}\right)\frac{1}{p}\left(|\nabla u_0|^{p-2}\nabla u_0 -|\nabla u_1|^{p-2}\nabla u_1\right)(\nabla u_0-\nabla u_1)\,dx\\ \nonumber &\geq  \int\limits_{1<p<2}(1-\frac{1}{p})|\nabla u_0-\nabla u_1|^2\left(1+|\nabla u_0|^2+|\nabla u_1|^2\right)^{\frac{p-2}{2}}\,dx\\ \nonumber &+\int\limits_{p\geq 2}
			\frac{1}{p\gamma(p)}|\nabla u_0 -\nabla u_1|^p\,dx
		\end{align}
		A fortiori, then, $u_0$ and $u_1$ must coincide.
	\end{proof}
	In particular, setting $q=0$ one obtains:
	
	\begin{corollary}\label{EU}
		Under the assumptions of Theorem \ref{main}, there exists a unique solution $u\in W^{1,p}(\Omega)$ to the boundary value problem
		\begin{equation}\label{DP}
			\begin{cases}
				\Delta_p(u)=\text{div}\left(|\nabla u|^{p-2}\nabla u\right)=f \,\,\,\text{in}\,\,\, \Omega\\
				u|_{\partial \Omega}= \varphi.
			\end{cases}
		\end{equation}
	\end{corollary}
	
	Setting $\varphi=0$ in Corollary \ref{EU}, one obtains Theorem 4.2 in \cite{FZ} for $f(x,u)=f(u).$

	\section*{Acknowledgements}
	The first author received support, from Khalifa University, UAE, through grant number 8474000357. The second author expresses his gratitude to the Department of Mathematical Sciences at Khalifa University for their kind hospitality throughout the project's implementation.\\
	
	\section*{Declaration of generative AI and AI-assisted technologies in the writing process}
	
	During the preparation of this work the authors used CHAT GPT in order to enhance the quality and fluency of the English writing style. After using this tool the authors reviewed and edited the content as needed and take full responsibility for the content of the publication.

\end{document}